\documentclass[10pt]{amsart}
\usepackage{amsmath, amsthm, amsfonts, amssymb}
\usepackage{enumerate}
\usepackage[bookmarks=false]{hyperref}

\usepackage{xcolor}

\usepackage{verbatim}

\usepackage{comment}

\newtheorem{thm}{Theorem}[section]

\newtheorem{col}[thm]{Corollary}
\newtheorem{lem}[thm]{Lemma}

\theoremstyle{definition}
\newtheorem{defn}[thm]{Definition}
\newtheorem{defn/lem}[thm]{Definition/Lemma}
\newtheorem{rmk}[thm]{Remark}

\newcommand{\C}{\ensuremath{\mathbb{C}}}

\newcommand{\M}{\ensuremath{\mathbb{M}}}
\newcommand{\N}{\ensuremath{\mathbb{N}}}
\newcommand{\Z}{\ensuremath{\mathbb{Z}}}

\newcommand{\cU}{\ensuremath{\mathcal{U}}}
\newcommand{\cT}{\ensuremath{\mathcal{T}}}

\begin{document}

\title[A new source of purely finite matricial fields]{A new source of purely finite matricial fields}

\author[D. Gao]{David Gao}
\address{Department of Mathematical Sciences, UCSD, 9500 Gilman Dr, La Jolla, CA 92092, USA}\email{weg002@ucsd.edu}\urladdr{https://sites.google.com/ucsd.edu/david-gao}

\author[S. Kunnawalkam Elayavalli]{Srivatsav Kunnawalkam Elayavalli}
\address{Department of Mathematics, UMD, Kirwan Hall, Campus Drive, MD 20770, USA}\email{sriva@umd.edu}
\urladdr{https://sites.google.com/view/srivatsavke/home}

\author[A. Manzoor]{Aareyan Manzoor}
\address{Department of Mathematics, University of Waterloo, Waterloo, Ontario, Canada}\email{amanzoor@waterloo.edu}
\urladdr{https://aareyanmanzoor.github.io}

\author[G. Patchell]{Gregory Patchell}
\address{\parbox{\linewidth}{Mathematical Institute, University of Oxford, Andrew Wiles Building, \\ Radcliffe Observatory Quarter, Woodstock Road, Oxford, OX2 6GG, UK}}
\email{greg.patchell@maths.ox.ac.uk}
\urladdr{https://sites.google.com/view/gpatchel}

\begin{abstract}
    A countable group $G$ is said to be \emph{matricial field} (MF) if it admits a strongly converging sequence of approximate homomorphisms into matrices; i.e, the norms of polynomials converge to those in the left regular representation. $G$ is \emph{purely MF} (PMF) if these maps are actual homomorphisms, and $G$ is further \emph{purely finite field} (PFF) if the image of each homomorphism is finite. By developing a new operator algebraic approach to these problems, we are able to prove the following result bringing several new examples into the fold. Suppose $G$ is a MF (resp., PMF, PFF) group and $H<G$ is separable (i.e., $H=\cap_{i\in \N}H_i$ where $H_i<G$ are finite index subgroups) and $K$ is a residually finite MF (resp., PMF, PFF) group. If either $G$ or $K$ is exact, then the amalgamated free product $G*_{H}(H\times K)$ is MF (resp., PMF, PFF). Our work has several applications, we list some below:
    \begin{enumerate}
    \item The Brown--Douglas--Fillmore semigroups of many new examples of reduced group $C^*$-algebras are shown to be not groups.
    \item Arbitrary group doubles $G*_HG$ of MF (resp., PMF, PFF) over separable subgroups $H$ are MF (resp., PMF, PFF). Moreover, $G*H$ is PFF whenever $G,H$ are PFF, and either $G$ or $H$ is exact. 
    \item Arbitrary graph products of residually finite exact MF (resp., PMF, PFF) groups are MF (resp., PMF, PFF), yielding a significant generalization of the breakthrough work of M. Magee and J. Thomas.
    \item The open problem of proving PFF for fundamental groups of closed hyperbolic 3-manifolds is resolved. This has geometric significance in the theory of minimal surfaces via A. Song's approach. 
    \end{enumerate}
\end{abstract}
\maketitle

\begin{center}
    \textit{for Vidhya Ranganathan}
\end{center}
\section{Introduction}

\subsection{Matricial field}

Understanding finite dimensional approximations for operator algebras allows one to effectively import intuition from the behavior in finite dimensions to understand infinite dimensional operator algebras and vice versa. An excellent example of the forwards direction is the recent revolutionary work on the Connes embedding problem \cite{Ji_2021_MIP_RE}. On the other hand, Voiculescu's introduction of free probability theory motivated by his study of the free group factors, in particular his cornerstone asymptotic freeness result \cite{Voiculescu1991, voiculescu1992free}, offers great testimony to the power of the backwards direction. The availability of matrix models that converge weakly to tracial noncommutative distributions in von Neumann algebras has also yielded striking insights into their structure, via free entropy theory \cite{PicuSurvey, HayesPT}. The subsequent groundbreaking work of Haagerup and Thorbjørnsen \cite{HaagerupThorbjornsen2005} opened up the possibility of extending the ``weak convergence'' of matrix models into the setting of $C^*$-algebras where one would demand ``strong convergence'' in operator norms. The endeavor of establishing strong convergence has recently gained prominence in both random matrix theory and operator algebras, and has paved the way for solutions to several long standing problems in multiple  areas of mathematics and applied mathematics \cite{vanhandel2025strongconvergenceshortsurvey, magee2025strong}.

Our present article furthers the study of strong convergence. The \emph{matricial field} (MF) property of countable groups is due to Blackadar and Kirchberg \cite{BlackadarKirchberg1997}. This property asks for finite dimensional approximations; i.e., matrix models for reduced group $C^*$-algebras in the strong sense. To be precise, we say $G$ is MF if for all finite sets $F \subseteq G$ and $\epsilon > 0$, there are $d \in \N$ and a function $u: G \rightarrow {U}(\M_d(\mathbb{C}))$ with
	\begin{enumerate}
		\item\label{def:MF-group1} $\|u_{gh} - u_gu_h\| < \epsilon$ for all $g, h \in F$,
		\item\label{def:MF-group2} $|\tau(u_g)| < \epsilon$ for all $g \in F \setminus\{1\}$, and
 		\item\label{def:MF-group3} $\big|\big\| \sum\limits_{g \in F} c_g u_g \big\| - \big\| \sum\limits_{g \in F} c_g \lambda_g\big\|_{C^*_r(G)} \big| < \epsilon$ for all $(c_g)_{g \in F} \subseteq \C$ with $\max\limits_{g \in F} |c_g| \leq 1$, where $\lambda_g$ refers to the unitary associated to $g\in G$ in the left regular representation.
	\end{enumerate}

 In other words, $G$ is MF if there exists an ultrafilter $\cU$ on $\N$ and a trace-preserving $*$-homomorphism from the reduced $C^*$-algebra, $\pi: C^*_r(G)\to \prod_\cU \M_n(\C)$, where the target $C^*$-algebra is the ultraproduct of matrix algebras with respect to the operator norm. Furthermore, $G$ is said to be \emph{purely MF} (PMF) if the maps $u$ from above can be chosen as actual homomorphisms. $G$ is said to be a \emph{purely finite field} (PFF) if additionally the image of the homomorphisms are finite at each stage. $G$ is further said to be a \emph{purely permutation field} (PPF) if in fact these homomorphisms are homomorphisms to finite permutation groups composed with standard irreducible representations of permutation groups. Note that the notations PMF and PPF are  borrowed from \cite{magee2025strong}. The relationship between these notions is transparent and is given below: 
$$\text{PPF} \implies \text{PFF} \implies \text{PMF}\implies  \text{MF}.$$

Outside of intrinsic interest in random matrix theory, these properties have phenomenal applications. MF has applications to operator algebras. If a $C^*$-algebra is MF but not quasidiagonal, then
$A$ has an extension by the compact operators $\mathbb{K}$ which is not invertible in the sense of Brown, Douglas and Fillmore \cite{BrownDouglasFillmore1977} (see for instance the remark at the end of Section 2 in \cite{Seebach2012}). In particular if a group is non amenable (hence not quasidiagonal \cite{Hadwin1987}) and MF, then $Ext(C^*_r(G))$ is not a group. PMF has certain geometric applications \cite{MageeThomas2023}, including spectral gaps of certain  
Laplacians acting on vector bundles. PFF has remarkably found use in the theory of minimal surfaces through work of A. Song in his approach \cite{song2025randomharmonicmapsspheres}. PFF is also interestingly in the spirit of soficity \cite{pestov2008hyperlinear} for algebras: it is not just a
finite-dimensional approximation of the reduced group $C^*$-algebra, but
even a finite approximation. PPF is the most desirable property here not only because it encompasses all of the above, but also because it plays a fundamental role in deep applications to the study of optimal spectral gaps of graphs and hyperbolic manifolds \cite{BordenaveCollins2019, HideMagee2023}. PPF is however quite difficult to access, especially because it fails in certain simple examples. One of the main novelties of the present article is that we demonstrate that it is still possible in significant generality to bypass PPF and access PFF instead. For further details and references concerning these notions we point the reader to the beautiful surveys \cite{magee2025strong, vanhandel2025strongconvergenceshortsurvey}.

Importantly, these properties are much harder to prove than weaker representability properties such as soficity or Connes embeddability \cite{pestov2008hyperlinear}, because it involves establishing {strong convergence}. Despite the vast collection of examples of soficity and hyperlinearity, at the moment only a handful of examples of groups are known to be MF/PMF/PFF/PPF. Moreover, it is noteworthy to point out that proving such properties in each of these cases has been profoundly difficult and needed deep insights. The current list of examples includes MF for amenable groups \cite{TWW2017}; PPF for free groups \cite{BordenaveCollins2019} (see also \cite{HaagerupThorbjornsen2005, Chen_2026}); PPF for limit groups \cite{loudermagee2025limitgroups}; PMF for right angled Artin groups (RAAGs) \cite{MageeThomas2023} (see also \cite{chen2024newapproachstrongconvergence}); MF for crossed products by free groups acting on amenable groups \cite{RainoneSchafhauser2019}; MF for $G_1*_HG_2$ where $G_1,G_2$ are amenable and $H$ is a normal subgroup \cite{schafhauser2024finite}. These properties are easily seen to be closed under taking subgroups and MF/PMF/PFF are closed under taking finite index overgroups (Lemma 7.1 in \cite{loudermagee2025limitgroups}, see also Lemma \ref{lem:strong-MF-facts}). Being closed under finite index overgroups seems open in the case of PPF. It is currently known that MF/PMF is also closed under free products \cite{collins2014strong, hayes2015lp}. Moreover, MF/PMF/PFF is also closed under direct products provided one of the groups is exact \cite{HadwinShen2010}. Strikingly, PPF is necessarily not closed under direct products, see \cite{magee2025strong} for a proof that $\mathbb{F}_2\times \mathbb{F}_2 \times \mathbb{F}_2$ is not PPF (an argument identical to Magee's shows $G\times H\times K$ is not PPF when $G,H$ are nonamenable and $K$ is nonabelian). Interestingly, it is also known that $SL_4(\mathbb{Z})$ is not PMF \cite{Magee_2024}. Despite all of these results, the situation for wider natural families of groups, in particular for amalgamated free products, has remained a challenging open problem. 

\subsection{Main results} In this paper we are able to make significant conceptual progress and expand on the collection MF, PMF and notably PFF groups by developing a new $C^*$-algebraic approach to the problem. Our main result is the following: 

\begin{thm}\label{thm:1.1}
    Suppose $G$ is an MF (resp., PMF, PFF) group and $H<G$ is separable (i.e., $H=\cap_{i\in \N}H_i$ where $H_i<G$ is a decreasing sequence of finite index subgroups). Let $L$ be a residually finite MF (resp., PMF, PFF) group such that either $G$ or $L$ is exact. Then the amalgamated free product $G*_{H}(H\times L)$ is MF (resp., PMF, PFF).
\end{thm}

Our approach is inspired by recent progress on \emph{selflessness} which is a form of intrinsic $C^*$-free independence in ultrapowers \cite{robertselfless, amrutam2025strictcomparisonreducedgroup, ozawa2025proximalityselflessnessgroupcalgebras}, in particular the work of Ozawa (see Section 4.6 of \cite{brown2008textrm}). In order to demonstrate the power of our result, we quickly assemble a few corollaries. Firstly, we observe that there is an isomorphism $G\ast_H (H\times\mathbb{Z}) \cong \ast_H G\rtimes \mathbb{Z}$ where $*_H G$ denotes the infinite group double of $G$ over $H$, indexed by elements of $\Z$, and $\Z$ acts on it by permuting the copies of $G$ in accordance with the left multiplication action of $\Z$ on itself. We also recall that, if $G$ and $H$ are PFF and one of them is exact, then $G \times H$ is PFF (Lemma \ref{lem:strong-MF-facts}). PFF groups are also automatically residually finite. This yields:

\begin{col}\label{col:1.2}
    Let $G$ be a MF (resp., PMF, PFF) group and $H<G$ be a separable subgroup. Then the group double $G*_HG$ is MF (resp., PMF, PFF). Moreover, if $G$ and $H$ are PFF groups such that one of them is exact, then $G*H$ is PFF. 
\end{col}

We point out that strikingly the case of free products had remained open for the PFF property until our result above. It turns out that conjugation by random unitaries as in \cite{collins2014strong} destroys the finiteness of the representations which is crucial in PFF.  Note also that the exactness assumption on one of the two groups is still necessary to apply our Theorem \ref{thm:1.1} in this situation, via passing to direct products.  In the context of the separability assumptions in Theorem \ref{thm:1.1} and Corollary \ref{col:1.2}, we remark that outside of the families of LERF groups (see a list of such examples in \cite{Gao2025}), a nice instance of separable inclusions $H<G$ which can be fed into our results is quasi-convex subgroups of cubulated hyperbolic groups. Indeed, cubulated hyperbolic groups are subgroups of RAAGs \cite{Agol2013} and hence PFF (our Corollary \ref{graph pdoructs cor} below), and quasi-convex subgroups are separable \cite{HaglundWise2008}. Another natural setting for separability is in the world of graph products \cite{HsuWise1999}. The following is our next main Corollary:

\begin{col}\label{graph pdoructs cor}
    Arbitrary graph products of exact residually finite MF (resp., PMF, PFF) groups are MF (resp., PMF, PFF). 
\end{col}

Our result above offers a significant generalization of the breakthrough work of Magee--Thomas \cite{MageeThomas2023}. We emphasize that this \emph{combination result} yields many new examples of MF and PMF groups, and even in the case of RAAGs, our proof of PMF is new and succinct. More importantly, the reach of our results is higher and can handle the PFF property for RAAGs and therefore all groups that embed virtually in them. For fundamental groups of closed hyperbolic 3-manifolds, this has been an open problem in the field\footnote{We thank Ramon van Handel for sharing with us the statement of this open problem.}, which is now settled. This particular result of ours is of geometric significance, as we shall explain soon, because PFF (for free and surface groups) is exactly the ingredient that goes into the breakthrough work of A. Song on minimal surfaces in Euclidean unit spheres \cite{song2025randomharmonicmapsspheres}. A. Song's paper accesses PFF via the PPF result of Bordenave--Collins \cite{BordenaveCollins2019}, and the follow up paper \cite{ancona2025minimalsurfacesnegativecurvature} accesses PFF via Louder--Magee's PPF for surface groups \cite{loudermagee2025limitgroups}. However, our results open the door to large new families of groups without having to go through PPF.   We document our result below (see \cite{Agol2013, HaglundWise2008} and also \cite[Theorem 5.27]{aschenbrenner2015manifold}): 

\begin{col}\label{Song cor}
Let $N$ be a non positively curved, compact, orientable,
aspherical 3-manifold with possibly empty boundary. Then $\pi_1(N)$ is PFF. More generally, all virtually special groups are PFF.
\end{col}

We are grateful to A. Song for sharing his helpful insights regarding the potential of the above result for geometry, which we briefly explain for the benefit of the reader before concluding this section. Given a closed manifold $M$ with fundamental group $G$, and given an $n$-dimensional unitary representation $\rho_n$ of $G$ such that $\rho_n(G)$ is finite, there is a natural way to construct a corresponding closed minimal submanifold $\Sigma_n$ inside some unit sphere. Finiteness, granted by the PFF property, is crucial here since the objects of interest are \emph{closed} submanifolds. When such a construction is fully justified, it turns out that PFF implies convergence of the shape of $\Sigma_n$ to a limit shape which can be deciphered, thanks to a good understanding of the regular representation. In the case of dimension $2$, closed 2-dimensional minimal surfaces that are almost hyperbolic have been constructed very recently in this manner \cite{song2025randomharmonicmapsspheres, ancona2025minimalsurfacesnegativecurvature}, answering old problems of S. T. Yau. Our Corollary \ref{Song cor} suggests strongly that this phenomena should carry over in dimension $3$ for instance, and more generally whenever $G$ virtually embeds in a RAAG. It is plausible that there are many closed minimal submanifolds $\Sigma^d$ of dimension $d \geq 3$ inside Euclidean spheres, whose limit geometry is that of a hyperbolic $d$-space. This would be satisfying from a geometric viewpoint as it would give the existence, in any dimension, of large families of closed minimal submanifolds in spheres whose limit shape is completely understood and canonical.

\subsection{Comments to the reader}

\subsection*{Optimality}\label{optimality} We would like to first discuss optimality aspects of our work. First note that exactness is a rather essential ingredient in certain cases, for instance Corollary \ref{graph pdoructs cor}. Indeed even in the tensor product case it is not known if MF is preserved in full generality \cite{HadwinShen2010}. A natural question arises of whether our main result and its corollaries can extend to accessing PPF. We point out that this is in fact not possible. For instance, our Corollary \ref{graph pdoructs cor} simply cannot generalize to PPF as stated because as mentioned earlier, Magee has shown that the group $\mathbb{F}_2\times \mathbb{F}_2\times \mathbb{F}_2$ is not PPF \cite{magee2025strong}, and this is indeed a RAAG. In fact, Magee's argument shows that $G\times H\times K$ is not PPF whenever $G,H$ are nonamenable and $K$ is nonabelian. The real challenge with PPF is that it is not amenable to being stable under direct products (or finite index overgroups), due to the following fundamental issue: tensoring does not preserve the image being in the orthogonal complement of the invariant vectors for permutation matrices. 

Another remark we would like to make is that it is at the moment an intractable open problem whether soficity or Connes embeddability is closed under amalgamated free products. In fact, considering the examples in Theorem \ref{thm:1.1}, it is unclear whether the separability assumption can be dropped, simply because even Connes embeddability is not known in that generality. 

\subsection*{Concreteness of matrix models}
At first glance, it may appear that our work reveals strongly convergent matrix models in a more abstract or existential sense. This is far from true. In fact, as we point out in the statement of Lemma \ref{lem:2.3}, we arrive at a very concrete and natural sequence of group homomorphisms
    \[G \ast_H (H \times \Z) \to [G/K_i\ast_{H_i/K_i} (H_i/K_i\times\Z)] \times G\] which sends $g\in G$ to $(gK_i, g)$ and $n\in \Z$ to $(n,e)$,
and argue that in fact these maps in the ultraproduct will extend to the reduced $C^*$-algebra. Note that the groups on the right hand side are products of virtually free groups (for a very concrete reason, see Lemma \ref{lem:2.4}) with $G$. Virtually free groups have a concrete strongly converging PFF model via \cite{BordenaveCollins2019} and an induced representation argument (see Lemma \ref{lem:strong-MF-facts}), and additionally $G$ is assumed to have a concrete PFF model. Combining these, we do in fact get concrete strongly converging models in our results, and in specific examples they can be tracked down explicitly.

\subsection*{A comment on bootstrapping}\label{last rmk} We include a minor remark that our work also provides an alternative proof of PFF for limit groups also (see Remark \ref{limit groups MF}). It is important to clarify however that the work \cite{loudermagee2025limitgroups} proves PPF for limit groups and this cannot be addressed currently by our method. The proof of this in our work combines our Theorem \ref{thm:1.1} and the fact that limit groups arise as iterated extensions of centralizers (Theorem 4.2 of \cite{ChampetierGuirardel2005}), and Hall's theorem for limit groups \cite{Wilton2008Halls}. We also point out for the benefit of the reader that the proof of Corollary \ref{graph pdoructs cor} follows from the fact that every graph product is constructed in an iterated fashion using the amalgamated free products in Theorem \ref{thm:1.1}, and additionally the amalgams are separable because they are retracts (Lemma 3.9 of \cite{HsuWise1999}). We also remark that our approach might also naturally be applied to get more examples of MF/PMF/PFF groups among the family of graph wreath products \cite{gao2024soficactionsgraphs}.

\subsection*{Insights into Theorem \ref{thm:1.1}} 
A key accomplishment of our work is to isolate a precise and powerful connection between some of the deep tools going into the emerging program on selflessness in $C^*$-algebras, and strong convergence of unitary representations. In effect, our approach is able to entirely reduce the problem of constructing strongly converging matrix models of groups of the form $G*_{H}(H\times \mathbb{Z})$, to just the case of free groups $\mathbb{F}_n$ (\cite{HaagerupThorbjornsen2005, BordenaveCollins2019}) and elementary constructions such as induced representations and tensor models via exactness. On one hand, handling such amalgamated free product groups allows us to cover large families at the same time thanks to beautiful results in geometric group theory around separability, RAAGs and virtual specialness \cite{Agol2013, HaglundWise2008, Wilton2008Halls, HsuWise1999}. On the other hand, this construction is precisely what allows us access to the deep $C^*$-correspondence theory machinery (see Theorem \ref{isomorphism thm}), in order to upgrade our maps from just the groups to the reduced $C^*$-algebras. 

Firstly, our maps from Lemma \ref{lem:2.3} are built into an ultraproduct of amalgamated free products (by passing to normal cores from the separability assumptions) which are each by design virtually free, and therefore admit various strongly converging models. In order for our maps to preserve norms, we use a Fell's absorption trick by tensoring with a copy of $G$ throughout the procedure. Interestingly, this is the reason why we need exactness in the case that we replace $\mathbb{Z}$ with an arbitrary residually finite group $L$, and also the reason why one cannot automatically use these arguments to build PPF models. Continuing with the argument, we crucially use a remarkable isomorphism theorem (Theorem \ref{isomorphism thm}), which says that a \emph{universal} Toeplitz--Pimnser type algebra satisfying just some simple orthogonality conditions, is automatically isomorphic to a \emph{reduced} amalgamated free product of the type we have been considering. All we need to do is check the prescribed equalities with the ``creation operators'', and we are then in business for strong convergence.

We would like to remark that the usage of isomorphism theorems of the above nature stand on tradition in operator algebras, as explained in Section 4.6 of \cite{brown2008textrm}.  For instance, it is used crucially in the proof of permanence of exactness for reduced amalgamated free products \cite{dykema2001exactness}, and additionally weak exactness \cite{toyosawa2025weak}. In the context of random matrices, the appendix of Shlyakhtenko in the paper \cite{male2011norm} features these ideas. More recently, a universality statement in the context of Toeplitz algebras associated to RAAGs, due to Crisp and Laca \cite{crisp2002toeplitz}, has been employed (albeit in combination with other deep random matrix theory methods \cite{Collins2022OnTO}) in the PMF work \cite{magee2025strong}. Our current work reveals in hindsight that these were in fact potent early insights of Shlyakhtenko, and more recently Magee--Thomas.  

\subsection*{Note to the reader} We strongly suggest the reader who is not familiar with the $C^*$-algebraic considerations herein to read our proof in parallel with Section 4.6 of \cite{brown2008textrm}, which has a very neat treatment of some constructions in $C^*$-correspondence theory, starting with the breakthrough of Pimsner \cite{pimsner1997class}, to the general \emph{gauge-invariant uniqueness theorem} \cite{FowlerMuhlyRaeburn2003} and its consequences including the isomorphism theorem we use here.

\subsection{Acknowledgments} 
We thank M. Junge for an insightful question that stimulated this work during the workshop ``Graduate mini-school on sofic groups and Connes embedding problem'' at UT Austin during March 18--22, 2026. We thank UT Austin and the organizers for providing a stimulating environment to work in. We are indebted to R. van Handel for several long correspondences and his extremely generous feedback and support. We are grateful to A. Song for his encouragement, and generous feedback.  We thank D. Jekel and N. Ozawa for helpful comments and corrections. We also thank B. Hayes, F. Fournier-Facio and D. Shlyakhtenko for helpful conversations regarding references and M. Linton for his suggestion of a quick proof of Lemma 2.5. The second author is grateful to U. Bader for his encouragement and helpful conversations. This work was supported partially by the second author's NSF grant DMS 2350049. The fourth author was supported by the Engineering and Physical Sciences Research Council (UK), grant EP/X026647/1. 

\subsection{Open Access and Data Statement}

For the purpose of Open Access, the authors have applied a CC BY public copyright license to any Author Accepted Manuscript (AAM) version arising from this submission. Data sharing is not applicable to this article as no new data were created or analyzed in this work.

\section{Proof of main result}

We will use standard notations in $C^*$-algebras, in particular Chapter 4 of the book \cite{brown2008textrm}, which is a great reference for a reader who is familiar with basic aspects of $C^*$-algebra theory. Unless otherwise mentioned, amalgamated free products of $C^*$-algebras equipped with conditional expectations will be in the reduced sense \cite{voiculescu1985symmetries}\footnote{We alert the reader to not be confused with MF for full $C^*$-algebras or for full amalgamated free products (see for instance \cite{shulman2026mfpropertyamalgamatedfree}) as the problems and techniques in that situation are of a very different nature.}. Additionally, all ultraproducts will be with respect to the operator norm, and all tensor products will be minimal. We recall the following for notational purposes. 
\begin{defn}
    Let $B\subset A$ be an inclusion of $C^*$-algebras, and $E:A\to B$ be a faithful conditional expectation. Then define the \emph{Toeplitz--Pimsner algebra} of $E$ to be the universal $C^*$-algebra
    \[\cT(E) := C^*\langle a\in A\,,T \big| \,T^*T=1,\, T^*aT = E(a)\rangle. \]

    $T$ is called the \emph{creation operator}. We recall that there is a canonical non-degenerate conditional expectation $E': \cT(E) \to A$ satisfying $E'(a_1Tb_1b_2^\ast T^\ast a_2^\ast) = 0$ for all $a_i, b_i \in A$ \cite[Theorem 4.6.6(2)]{brown2008textrm}.

    A special case of the above is the \emph{classical Toeplitz algebra}, which we will denote by $\cT$. It corresponds to the case where $A = B = \C$ and $E = \text{id}_\C$. An alternative description is that $\cT$ is the universal $C^\ast$-algebra of an isometry. In this case, the canonical conditional expectation yields a non-degenerated state $\omega$ on $\cT$, which is called the \emph{vacuum state}, and we shall always understand $\cT$ as equipped with this state.
\end{defn}

The following Theorem \ref{isomorphism thm} in essence goes back to the breakthrough of Pimsner \cite{pimsner1997class}, but can be concretely deduced from the gauge-invariant uniqueness \cite{FowlerMuhlyRaeburn2003} combined with the observation that the concrete Toeplitz-Pimsner algebra defined on the Fock space representation  is the reduced amalgamated free product \cite{shlyakhtenko1999valued}. For the convenience of the reader we provide a short argument (following \cite{brown2008textrm} wherein this result also appears). Our inspiration for employing this result in our context comes from the recent outburst of developments surrounding selfless $C^*$-algebras (\cite{robertselfless, amrutam2025strictcomparisonreducedgroup, raum2025strictcomparisontwistedgroup, HKER, ozawa2025proximalityselflessnessgroupcalgebras, hayes2025selfless, relativeselfless, elayavalli2025negativeresolutioncalgebraictarski, vigdorovich2025structural, vigdorovich2026selflessreducedcalgebraslinear, flores2026purenessstablerankreduced, avni2025mixedidentitieslineargroups, BradfordSisto2026,  yang2025extreme,  fmmm2025selflessfreeprod, selflessAFPs}), especially the proof technique of Ozawa in \cite{ozawa2025proximalityselflessnessgroupcalgebras} and also the upcoming key work \cite{relativeselfless}. The first, second and fourth authors are grateful to M. Junge and L. Robert for the collaboration on \cite{relativeselfless}.

%Before stating and proving Theorem \ref{isomorphism thm}, we emphasize that the beautiful gauge-invariant uniqueness theorem (this is explicitly \cite[Theorem 4.6.18]{brown2008textrm}) is used in the proof of Theorem \ref{isomorphism thm}. We highly encourage the reader to look at Section 4.6 in the treatise \cite{brown2008textrm} which offers a clean treatment of these aspects.

\begin{thm}\label{isomorphism thm}
    Let the notation be as above. Then the $C^*$-subalgebra generated by $A$ and $\frac{1}{2}(T+T^*)$ in $\cT(E)$ is isomorphic to $A\ast_{B} (B\otimes C[-1,1])$, where $\frac{1}{2}(T+T^*)$ is the semicircular in $C[-1,1]$.
\end{thm}
\begin{proof}
    First consider $A\ast_B (B\otimes \cT)$ where $\cT$ is the classical Toeplitz algebra. Let $S$ be the generating isometry in $\cT$. Then for $a\in A$, $S^*aS= E(a)$. Hence there is a map $\pi:\cT(E)\to A\ast_B (B\otimes \cT)$ which restricts to identity on $A$ and sends $T\mapsto S$. This map is surjective as $A\ast_B (B\otimes \cT)$ is generated by $A$ and $\cT$.

    We will prove $\pi$ is injective using the \emph{gauge-invariant uniqueness theorem} for Toeplitz--Pimsner algebras \cite[Theorem 4.6.18]{brown2008textrm} (this is originally due to \cite{FowlerMuhlyRaeburn2003}). For this, note that the canonical conditional expectation $E':A\ast_B (B\otimes \cT)\to A$ acts by $\operatorname{id}\otimes \omega$ on the second factor, where $\omega$ is the vacuum state on the classical Toeplitz algebra. So we have, for $a \in A$,
    \[E'(SaS^*) = E'(SE(a)S^*)+E'(S(a-E(a))S^*)= E(a)E'(SS^*)=0,\]
    where we used that $S$ commutes with $E(a)\in B$. Thus, $\pi(A)\cap \overline{\operatorname{span}}\{a_1 Sb_1 b_2^*S^* a_2^*:a_1,b_1,b_2,a_2 \in A\}=\{0\}$, and hence $\pi$ satisfies the hypothesis of the gauge-invariant uniqueness theorem.

    Restricting to the subalgebra generated by $A$ and $\frac{1}{2}(T+T^*)$ in $\cT(E)$ gives the result.
\end{proof}

The following is the key argument in our proof. 

\begin{lem}\label{lem:2.3}
    Let $G$ be a group and $H$ a subgroup with $H=\bigcap_{i\in \N} H_i$ for some decreasing sequence of subgroups $H_i$. Let $K_i = \bigcap_{g \in G} gH_ig^{-1}$ be the normal core of $H_i$. Let $\cU$ be a free ultrafilter on $\N$. Then there is a trace-preserving embedding of $C^*$-algebras:
    \[C_r^*(G\ast_H (H\times\Z)) \hookrightarrow \prod_\cU C_r^*([G/K_i\ast_{H_i/K_i} (H_i/K_i\times\Z)] \times G).\]

    Moreover, this embedding lifts to a sequence of group homomorphisms
    \[G \ast_H (H \times \Z) \to [G/K_i\ast_{H_i/K_i} (H_i/K_i\times\Z)] \times G\] which sends $g\in G$ to $(gK_i, g)$ and $n\in \Z$ to $(n,e)$. 

\end{lem}

\begin{proof}
    Consider the sequence of group homomorphisms
    \begin{equation*}
        G \ast_H (H \times \Z) \to [G/K_i \ast_{H_i/K_i} (H_i/K_i \times \Z)] \times G
    \end{equation*}
    defined by sending $g\in G$ to $(gK_i, g)$ and $n\in \Z$ to $(n,e)$. It is easy to check that this sequence of group homomorphisms is trace-preserving in the limit, so it induces a trace-preserving embedding,
    \begin{equation}\label{eqn: group alg map}
    \begin{split}
        \C[G \ast_H (H \times \Z)] \hookrightarrow &\prod_\cU C_r^\ast([G/K_i \ast_{H_i/K_i} (H_i/K_i \times \Z)] \times G)\\
        = &\prod_\cU C_r^\ast(G/K_i \ast_{H_i/K_i} (H_i/K_i \times \Z)) \otimes C_r^\ast(G).
    \end{split}
    \end{equation}

    It suffices to show this embedding is continuous under the reduced norm. By Fell's absorption principle, this is indeed isometric on $\C[G]$ and thus induces an embedding,
    \begin{equation}\label{eqn: embedding}
        C_r^\ast(G) \hookrightarrow \prod_\cU C_r^\ast(G/K_i \ast_{H_i/K_i} (H_i/K_i\times\Z)) \otimes C_r^\ast(G).
    \end{equation}

    Let $E:C_r^\ast(G) \to C_r^\ast(H)$ and $E_i: C_r^\ast(G/K_i) \to C_r^\ast(H_i/K_i)$ be the respective conditional expectations. Let $q_i: \C[G] \to C_r^\ast(G/K_i)$ be the map induced from the quotient map $G \to G/K_i$. Identifying the image of the embedding above with $C_r^\ast(G)$, it is easy to check that the conditional expectation $E$ under this identification sends $(q_i(\lambda_g) \otimes \lambda_g)_\cU$ to $(q_i(E(\lambda_g)) \otimes \lambda_g)_\cU$ for $g \in G$.

    Fix a trace-preserving embedding $C_r^\ast(\Z) \subset C[-1, 1]$. Then,
    \begin{equation}\label{eqn: group embedding}
    \begin{split}
        &C_r^\ast(G/K_i \ast_{H_i/K_i} (H_i/K_i\times\Z)) \otimes C_r^\ast(G)\\
        \subset &[C_r^\ast(G/K_i) \ast_{C_r^\ast(H_i/K_i)} (C_r^\ast(H_i/K_i) \otimes C[-1,1])] \otimes C_r^\ast(G).
    \end{split}
    \end{equation}

    Per Theorem \ref{isomorphism thm}, we furthermore have $C_r^\ast(G/K_i) \ast_{C_r^\ast(H_i/K_i)} (C_r^\ast(H_i/K_i) \otimes C[-1,1]) \subset \cT(E_i)$, so,
    \begin{equation}\label{eqn: embedding 2}
    \begin{split}
        &\prod_\cU C_r^\ast(G/K_i \ast_{H_i/K_i} (H_i/K_i\times\Z)) \otimes C_r^\ast(G)\\
        \subset &\prod_\cU [C_r^\ast(G/K_i) \ast_{C_r^\ast(H_i/K_i)} (C_r^\ast(H_i/K_i) \otimes C[-1,1])] \otimes C_r^\ast(G)\\
        \subset &\prod_\cU \cT(E_i) \otimes C_r^\ast(G),
    \end{split}
    \end{equation}
    where we note that the embedding acts as the identity map on the tensor component $C_r^\ast(G)$, restricts to the natural embedding $C_r^\ast(G/K_i) \hookrightarrow \cT(E_i)$, and sends the semicircular generator of $C[-1, 1]$ to $\frac{1}{2}(T_i + T_i^\ast)$ in $\cT(E_i)$.
    
    Now note that, for $g \in G$,
    \begin{equation*}
        (T_i^\ast q_i(\lambda_g)T_i \otimes \lambda_g)_\cU = (E_i(q_i(\lambda_g)) \otimes \lambda_g)_\cU = (q_i(E(\lambda_g)) \otimes \lambda_g)_\cU
    \end{equation*}
    where the latter equality can be verified easily using the fact that $H_i$ decreases to $H$. Indeed, if $g \in H$, then $q_i(\lambda_g) = \lambda_{gK_i} \in C_r^\ast(H_i/K_i)$, so
    \begin{equation*}
        E_i(q_i(\lambda_g)) \otimes \lambda_g = q_i(\lambda_g) \otimes \lambda_g = q_i(E(\lambda_g)) \otimes \lambda_g.
    \end{equation*}

    If, on the other hand, $g \notin H$, then for large enough $i$, $g \notin H_i$. Thus, $gK_i \notin H_i/K_i$. Hence, again for large enough $i$,
    \begin{equation*}
        E_i(q_i(\lambda_g)) \otimes \lambda_g = 0 = q_i(E(\lambda_g)) \otimes \lambda_g.
    \end{equation*}
    
    This proves the claimed equality. So, by taking linear combinations and norm limits, this implies the copy of $C_r^\ast(G)$ in $\prod_\cU \cT(E_i) \otimes C_r^\ast(G)$, via the embeddings in Equations (\ref{eqn: embedding}) and (\ref{eqn: embedding 2}), together with the isometry $(T_i)_\cU$, satisfies the universal property defining $\cT(E)$. Hence, there is a $\ast$-homomorphism
    \begin{equation*}
        \cT(E) \to \prod_\cU \cT(E_i) \otimes C_r^*(G)
    \end{equation*}
    which extends the embedding of $C_r^\ast(G)$ into $\prod_\cU \cT(E_i) \otimes C_r^\ast(G)$ and sends $T$ to $(T_i)_\cU$. By Theorem \ref{isomorphism thm}, restricting to the algebra generated by $C_r^\ast(G)$ and $\frac{1}{2}(T + T^\ast)$, we then have a $\ast$-homomorphism,
    \begin{equation}\label{eqn: semicircular map}
    \begin{split}
        &C_r^\ast(G) \ast_{C_r^\ast(H)} (C_r^\ast(H) \otimes C[-1,1])\\
        \to &\prod_\cU [C_r^\ast(G/K_i) \ast_{C_r^\ast(H_i/K_i)} (C_r^\ast(H_i/K_i) \otimes C[-1,1])] \otimes C_r^\ast(G).
    \end{split}
    \end{equation}

    This map sends $\lambda_g \in C_r^\ast(G)$ to $\lambda_{gK_i} \otimes \lambda_g$ and restricts to the identity map on $C[-1, 1]$, as it sends $\frac{1}{2}(T + T^\ast)$ to $\left(\frac{1}{2}(T_i + T_i^\ast)\right)_\cU$.

    Again, using the previously fixed trace-preserving embedding $C_r^\ast(\Z) \subset C[-1, 1]$ as in Equation (\ref{eqn: group embedding}), we can restrict to a $\ast$-homomorphism,
    \begin{equation*}
    \begin{split}
        &C_r^\ast(G \ast_H (H \times \Z))\\
        = &C_r^\ast(G) \ast_{C_r^\ast(H)} (C_r^\ast(H) \otimes C_r^\ast(\Z))\\
        \hookrightarrow &\prod_\cU [C_r^\ast(G/K_i) \ast_{C_r^\ast(H_i/K_i)} (C_r^\ast(H_i/K_i) \otimes C_r^\ast(\Z))] \otimes C_r^\ast(G)\\
        = &\prod_\cU C_r^\ast([G/K_i \ast_{H_i/K_i} (H_i/K_i \times \Z)] \times G).
    \end{split}
    \end{equation*}

    Since the map in Equation (\ref{eqn: semicircular map}) acts as the identity map on $C[-1, 1]$, the map above acts as the identity map on $\Z$. Moreover, since the map in Equation (\ref{eqn: semicircular map}) sends $\lambda_g \in C_r^\ast(G)$ to $\lambda_{gK_i} \otimes \lambda_g$, the map above sends $g \in G$ to $(gK_i, g) \in [G/K_i \ast_{H_i/K_i} (H_i/K_i \times \Z)] \times G$. This means the map above is exactly the extension of the embedding in Equation (\ref{eqn: group alg map}), which proves that the embedding there is indeed continuous under the reduced norm, as claimed.
\end{proof}

We note the following fact which is certainly well known to experts, see for instance Lemma 7.1 of \cite{loudermagee2025limitgroups}. For the benefit of the reader we include a proof.

\begin{lem}\label{lem:strong-MF-facts}
    \;
    \begin{enumerate}
        \item Virtually free groups are PFF;
        \item If $G$ and $H$ are PMF and $G$ is exact, then $G \times H$ is PMF. If they are in addition PFF, then the product is PFF as well.
    \end{enumerate}
\end{lem}

\begin{proof}
    For (1), free groups are PPF by \cite{BordenaveCollins2019} and so PFF. The result follows by considering induced representations. To be precise, assume $H$ is PFF and $G$ contains $H$ as a finite-index subgroup. Then we may take a sequence of strongly converging representations $\sigma_n: H \to M_{d(n)}(\C)$ that quotients through finite groups $H_n$. Then the induced representations $\phi_n: G \to M_{[G: H]}(M_{d(n)}(\C))$ converge strongly. Furthermore, the range of $\phi_n$ is contained within the following set of $[G:H]$-by-$[G:H]$ matrices with entries in $M_{d(n)}(\C)$:
        $\{A \in M_{[G: H]}(M_{d(n)}(\C)):\text{there is a permutation matrix }P \in M_{[G: H]}(\C)\text{ s.t. } A_{ij} = 0\text{ whenever }P_{ij} = 0\text{ and }A_{ij} \in \operatorname{range}(\sigma_n)\text{ whenever }P_{ij} = 1\}.$

    It is easy to see that the set above is finite, so the result follows.

    Item (2) for PMF follows by noting that, as $G$ is exact, the following is short exact sequence,
    \begin{equation*}
        0 \to \bigoplus_n M_{d(n)}(C_r^*(G)) \to \left(\prod_n M_{d(n)}(\C)\right) \otimes C_r^*(G) \to \frac{\prod_n M_{d(n)}(\C)}{\bigoplus_n M_{d(n)}(\C)} \otimes C_r^*(G) \to 0.
    \end{equation*}

    The middle term embeds into $\prod_n M_{d(n)}(C_r^*(G))$ naturally. Since $C_r^*(H)$ embeds into $\frac{\prod_n M_{d(n)}(\C)}{\bigoplus_n M_{d(n)}(\C)}$ in a way that can be lifted to a sequence of group homomorphisms, this implies $C_r^*(G \times H)$ embeds into
    \begin{equation*}
        \frac{\prod_n M_{d(n)}(C_r^*(G))}{\bigoplus_n M_{d(n)}(C_r^*(G))}
    \end{equation*}
    in a way that can be lifted to a sequence of group homomorphisms. Now, apply the assumption that $G$ is PMF to obtain the result.

    Item (2) for PFF follows by noting that, in the above, the finite-dimensional representations obtained that strongly converge to $G \times H$ is a sequence of tensor representations. More precisely, if $\sigma_m: G \to U(d_1(m))$ and $\phi_k: H \to U(d_2(k))$ converge strongly to $G$ and $H$, respectively, then there exist sequences of natural numbers $m(n)$ and $k(n)$ s.t. $\sigma_{m(n)} \otimes \phi_{k(n)}: G \times H \to U(d_1(m(n))d_2(k(n)))$ converges strongly to $G \times H$. Note that if $\sigma_m$ and $\phi_k$ factor through finite quotients, so does $\sigma_m \otimes \phi_k$. The result follows.
\end{proof}

The following is elementary.

\begin{lem}\label{lem:2.4}
    Let $G$ be a finite group and $H<G$. Then $G\ast_{H} (H\times \Z)$ is virtually free and thus in particular PFF.
\end{lem}
\begin{proof}
   Define a homomorphism $\pi: G*_H (H\times \Z)\to G$ given by the identity map on $G$ and sending $\Z$  to 1. Note that the kernel of $\pi$ intersected with any conjugate of $G$ is the trivial group. Thus, the action of $\ker(\pi)$ on the Bass--Serre tree gives a splitting as a graph of groups with trivial edge groups and with some trivial vertex groups and some $\Z$ vertex groups \cite{Ser80}. In particular it must actually be a free group. Hence $G$ is virtually free, which implies PFF by Lemma \ref{lem:strong-MF-facts}. 
\end{proof}

Now we can prove our main theorem:
\begin{proof}[Proof of Theorem \ref{thm:1.1}]
    We first prove the case $L = \mathbb Z$. As in Lemma \ref{lem:2.3}, we set $K_i = \bigcap_{g \in G} gH_ig^{-1}$, the normal core of $H_i$. Note that $K_i$ is of finite index in $G$, so we have $G/K_i\ast_{H_i/K_i} (H_i/K_i\times\mathbb{Z})$ is PFF by Lemma \ref{lem:2.4}. These are exact since amalgamated free products of exact groups are exact \cite[Theorem 3.2]{Dykema2004Exactness}. Since $G$ is MF (resp., PMF, PFF) by hypothesis, we have that \[[G/K_i\ast_{H_i/K_i} (H_i/K_i\times\mathbb{Z})]\times G\] is also MF (resp., PMF, PFF) by \cite[Prop 3.2]{HadwinShen2010} or Lemma \ref{lem:strong-MF-facts}. Now the result follows from Lemma \ref{lem:2.3}.

    Now let $L$ be a general residually finite, MF (resp., PMF, PFF) group. Since either $G$ or $L$ is exact and $G$ is MF (resp., PMF, PFF), $G\times L$ is MF (resp., PMF, PFF). Since $L$ is residually finite, $\{e\}$ is an intersection of finite index subgroups of $L$, so $H\times \{e\} \subset G\times L$ is separable. By the previous paragraph, $(G\times L) *_{H\times\{e\}}(H\times\{e\}\times \mathbb Z)$ is MF (resp., PMF, PFF). Now, note that we have the following embeddings:
    \begin{align*}
        G*_H(H\times L) &\hookrightarrow (G\times L) *_{H\times\{e\}} (G\times L) \\
        & \hookrightarrow \ast_{H\times\{e\}} (G\times L)\rtimes \mathbb{Z}\\
        &\cong (G\times L) *_{H\times\{e\}}(H\times\{e\}\times \mathbb Z).
    \end{align*}
    As subgroups of MF (resp., PMF, PFF) groups are MF (resp., PMF, PFF), this concludes the proof.
\end{proof}

\begin{proof}[Proof of Corollary \ref{graph pdoructs cor}]
    We proceed by induction on the number of vertices. The base case of one vertex is clear. Now suppose any graph product over a graph with $n-1$ vertices of exact, MF (resp., PMF, PFF), residually finite groups is MF (resp., PMF, PFF). Let $\Gamma=(V,E)$ be a  graph with $n$ vertices. Pick a vertex $v\in V.$ For a subset $W\subset V$, let $\Gamma_W$ be the subgroup of $\Gamma_{\mathcal G}$ generated by $\{G_w : w\in W\}$. For a vertex $v\in V,$ let $\mathrm{lk}(v)$ be the \emph{link} of $v$; that is, the vertices $w\in V$ such that $(v,w)\in E$. Recall that $\Gamma_{\mathcal G} = \Gamma_{V\setminus\{v\}} *_{\Gamma_{\mathrm{lk}(v)}}(\Gamma_{\mathrm{lk}(v)}\times G_v)$. To apply Theorem \ref{thm:1.1}, all of the conditions are immediately satisfied by the hypotheses except that ${\Gamma_{\mathrm{lk}(v)}}$ is separable. But this follows from the residual finiteness of all of the $G_w$, see Lemma 3.9 of \cite{HsuWise1999}. So the graph product $\Gamma_{\mathcal G}$ is MF (resp., PMF, PFF).
\end{proof}

For the benefit of the reader we also include the following alternative proof of PFF for limit groups.

\begin{rmk}\label{limit groups MF}
    Limit groups are PFF. 
\end{rmk}

\begin{proof}
    We first recall that all limit groups embed into an iterated extension of centralizers of a free group; i.e., if $G$ is a limit group then one can embed $G$ into a group $G_n$ obtained by starting with a free group $G_0$ and forming amalgamated free products of the form $G_{i+1} = G_i *_{A_i} (A_i\times \mathbb Z)$ for some cyclic subgroups $A_i$. See for instance Theorem 4.2 of  \cite{ChampetierGuirardel2005}. All such $G_i$ are LERF groups \cite{Wilton2008Halls}, so in particular each $A_i$ is separable. As free groups are PFF, the result follows by recursively applying Theorem \ref{thm:1.1}.
\end{proof}

\bibliographystyle{amsalpha}
\bibliography{bibliography}

\end{document}